\documentclass{amsart}
\usepackage{amssymb,darinams}
\usepackage[all]{xy}
\numberwithin{theorem}{section}

\moprm{\detrg}{detR\Gamma}
\moprm{\ir}{ir}
\moprm{\length}{length}
\moprm{\depth}{depth}
\mopit{\SPEC}{\cS pec}
\moprm{\Pic}{Pic}
\moprm{\Sym}{Sym}
\moprm{\Hilb}{Hilb}
\moprm{\HC}{HC}
\newcommand\IS{{\widetilde{\Hilb}}}
\moprm{\Flag}{Flag}
\moprm{\Heck}{Hecke}
\moprm{\coker}{coker}
\newcommand{\oPic}{\overline{\Pic}}
\moprm{\Fl}{Fl}

\newcommand\D {\mathbb{D}}
\newcommand\DC{\mathbb{D}\mathbb{C}}
\newcommand\FF{\mathfrak{F}}
\newcommand\FS{\mathfrak{FSch}}

\newcommand\GL{\mathrm{GL}}
\newcommand\Bun{\mathrm{Bun}}
\newcommand\HH{\mathbb{H}}
\moprm{\ch}{ch}
\moprm{\Td}{Td}
\newcommand\frj{{\mathfrak{j}}}

\newenvironment{myenum}{\setcounter{enumi}{0}\def\item{\par\refstepcounter{enumi}(\theenumi) }}{}

\newcommand\kkk{{\Bbbk}}

\begin{document}
\title[Autoduality of compactified Jacobians]{Autoduality of compactified Jacobians for curves with plane singularities}
\author{Dima Arinkin}
\address{Department of mathematics, University of North Carolina, Chapel Hill, NC}
\email{arinkin@email.unc.edu}
\begin{abstract} Let $C$ be an integral projective curve with planar singularities. Consider its Jacobian
$J$ and the compactified Jacobian $\oJ$. 
We construct a flat family $\oP$ of Cohen-Macaulay sheaves on $\oJ$ parametrized by $\oJ$; 
its restriction to $J\times\oJ$ is the Poincar\'e line bundle. We prove that the Fourier-Mukai transform given by
$\oP$ is an auto-equivalence of the derived category of $\oJ$.
\end{abstract}

\maketitle
\section*{Introduction}
Let $C$ be a smooth irreducible projective curve over a field $\kkk$, and let $J$ be the Jacobian of $C$. 
As an abelian variety, $J$ is self-dual.
More precisely, $J\times J$ carries a natural line bundle (the Poincar\'e bundle) $P$
that is universal as a family of topologically trivial line bundles on $J$.

The Poincar\'e bundle defines the Fourier-Mukai functor
$${\mathfrak F}:D^b(J)\to D^b(J):\cG\mapsto Rp_{2,*}(p_1^*(\cG)\otimes P).$$
Here $D^b(J)$ is the bounded
derived category of quasi-coherent sheaves on $J$ and $p_{1,2}:J\times J\to J$ are the projections.
Mukai proved that $\mathfrak F$ is an equivalence of categories (\cite{Mukai}).

Now suppose that $C$ is a singular curve, assumed to be projective and integral.
The Jacobian $J$ is no longer projective, but it admits a natural compactification $\oJ\supset J$ (\cite{CJ,irreducibility}). By definition,
$\oJ$ is the moduli space of torsion-free sheaves $F$ on $C$ such that $F$
has generic rank one and $\chi(F)=\chi(O_C)$;
$J$ is identified with the open subset of locally free sheaves.
It is natural to ask whether $\oJ$ is self-dual.

In this paper, we prove such self-duality assuming that $C$ is an integral projective curve with planar singularities
over a field $\kkk$ of characteristic zero. We construct a Poincar\'e sheaf $\oP$ on $\oJ\times\oJ$. The sheaf is
flat over each copy of $\oJ$; we can therefore view it as a $\oJ$-family of sheaves on $\oJ$. We prove that this
family is universal in the sense that it identifies $\oJ$ with a connected component of the moduli space of
torsion-free sheaves of generic rank one on $\oJ$. This generalizes autoduality results of \cite{autoduality,compactified,Ar}
and answers the question posed in \cite{compactified}.
We also prove that the corresponding Fourier-Mukai functor 
\[
{\mathfrak F}:D^b(\oJ)\to D^b(\oJ):\cG\mapsto Rp_{2,*}(p_1^*(\cG)\otimes \oP)
\]
is an equivalence.

\begin{remarks*}
\begin{myenum}
\item If $C$ is a plane cubic (nodal or cuspidal), these results are known: see the remark after 
Theorem~\ref{th:FM}.

\item It is easy to write a formula for the Poincar\'e line bundle $P$ on $J\times\oJ$: see \eqref{eq:Poincare}.
Our result is thus a construction of an extension of $P$ to a sheaf on $\oJ\times\oJ$ satisfying certain natural
properties.

We studied $P$ in \cite{Ar}. There, we prove weaker versions of the results of the present paper:
that $P$ is a universal family of topologically trivial line bundles on $\oJ$, and that the corresponding
Fourier-Mukai functor
\[{\mathfrak F}_J:D^b(J)\to D^b(\oJ)\]
is fully faithful. Note that 
${\mathfrak F}_J={\mathfrak F}\circ Rj_*$, where $j:J\hookrightarrow\oJ$ is the
open embedding.  

\item
Compactified Jacobians appear as (singular)
fibers of the Hitchin fibration for group $\GL(n)$; therefore, our results imply a kind of autoduality
of the Hitchin fibration for $\GL(n)$. Such autoduality can be viewed as a `classical limit' of the (conjectural)
Langlands transform. For arbitrary reductive group $G$, one expects a duality between the Hitching fibrations of $G$ and its Langlands dual ${}^LG$. For smooth fibers of the Hitchin fibration, such duality is proved by 
R.~Donagi and T.~Pantev in \cite{DP} (assuming some non-degeneracy conditions).

\item
Our construction of $\oP$ can be obtained as a `classical limit' of V.~Drinfeld's construction of automorphic
sheaves for group $\GL(2)$ (\cite{Dr}), see Section~\ref{sc:langlands} for more details.
\end{myenum}
\end{remarks*}

\subsection*{Acknowledgments}
I am very grateful to V.~Drinfeld for sharing his ideas on this subject. Thanks to him, I abandoned my original, much
clumsier, approach based on presentation schemes. 

I would also like to thank R.~Donagi, V.~Ginzburg, S.~Kumar, and T.~Pantev for their remarks and suggestions.

This work is supported in part by Alfred P.~Sloan Foundation under the Sloan Research Fellowship program.

\section{Main results}

\subsection{Summary of main results}
Fix a ground field $\kkk$ of characteristic zero. To simplify notation, we also assume that $\kkk$ is algebraically closed;
this assumption is not necessary for the argument.
Let $C$ be an integral projective curve over $\kkk$. Denote by $g$ the arithmetic genus of $C$, and let $J$ be the Jacobian
of $C$, that is, $J$ is the moduli space of line bundles on $C$ of degree zero. Denote by $\oJ$ the compactified
Jacobian; in other words, $\oJ$ is the moduli space of torsion-free sheaves
on $C$ of generic rank one and degree zero.
(For a sheaf $F$ of generic rank one, the degree is $\deg(F)=\chi(F)-\chi(O_C)$.)
Then $\oJ$ is an irreducible projective variety; it is locally a complete intersection of dimension $g$ (\cite{irreducibility}).
Clearly, $J\subset\oJ$ is an open smooth subvariety.

Let $P$ be the Poincar\'e bundle; it is a line bundle on $J\times\oJ$. Its fiber over $(L,F)\in J\times\oJ$
equals
\begin{equation}
\label{eq:Poincare}
P_{(L,F)}=\detrg(L\otimes F)\otimes\detrg (O_C)\otimes\detrg(L)^{-1}\otimes\detrg(F)^{-1}.
\end{equation}
More explicitly, we can write $L\simeq O(\sum a_ix_i)$ for a divisor $\sum a_ix_i$ supported by the smooth locus
of $C$, and then
$$P_{(L,F)}=\bigotimes(F_{x_i})^{\otimes a_i}.$$

The same formula \eqref{eq:Poincare} defines $P_{(F_1,F_2)}$ for any pair $(F_1,F_2)\in\oJ\times\oJ$ with either
$F_1\in J$ or $F_2\in J$. Equivalently, $P|_{J\times J}$ is symmetric under the permutation of factors in $J\times J$;
and therefore $P$ naturally extends to a line bundle on $J\times\oJ\cup\oJ\times J\subset\oJ\times\oJ$. We denote this
extension by the same letter $P$.

For the rest of the paper, we assume that $C$ has planar singularities; in other words,
the tangent space to $C$ at any point is at most two-dimensional.

\begin{THEOREM} \label{th:oP}
There exists a coherent sheaf $\oP$ on $\oJ\times\oJ$
with the following properties:
\begin{enumerate}
\item\label{th:oP1} $\oP|_{J\times\oJ\cup\oJ\times J}\simeq P$;
\item\label{th:oP3} 
$\oP$ is flat for the projection $p_2:\oJ\times\oJ\to\oJ$, and the restriction $\oP|_{\oJ\times\{F\}}$ is a Cohen-Macaulay
sheaf for every $F\in\oJ$.
\end{enumerate}
\end{THEOREM}

\begin{remark*} As explained in Section~\ref{sc:properties}, Theorem~\ref{th:oP} uniquely determines $\oP$ as a 
`Cohen-Macaulay extension' of $P$ under the embedding $j:J\times\oJ\cup\oJ\times J\hookrightarrow\oJ\times\oJ$. 
In fact, $\oP=j_*P$. 
\end{remark*}

By Theorem~\ref{th:oP}\eqref{th:oP3}, $\oP$ is a family of coherent (Cohen-Macaulay) sheaves on $\oJ$ parametrized by $\oJ$.
For fixed $F\in\oJ$, denote the corresponding coherent sheaf on $\oJ$ by $\oP_F$. In other words, $\oP_F$ is the restriction 
$\oP|_{\oJ\times\{F\}}$. 

Let $\Pic(\oJ)^=$ be the moduli space of torsion-free sheaves of generic rank one on $\oJ$. A.~Altman and S.~Kleiman proved in
\cite{CP1,CP2} that connected components of $\Pic(\oJ)^=$ are proper schemes (\cite[Theorem~3.1]{CP2}). 
The correspondence $F\mapsto\oP_F$ can be viewed as a morphism 
\[\rho:\oJ\to\Pic(\oJ)^=.\]
Denote by $\oPic^0(\oJ)\subset\Pic(\oJ)^=$ the irreducible component of the trivial bundle
$O_\oJ\in\Pic(\oJ)\subset\Pic(\oJ)^=$ (since $O_\oJ\in\Pic(\oJ)^=$ is a smooth point, 
it is contained in a single component). 
We prove that $\oJ$ is self-dual in the following sense.

\begin{THEOREM}\label{th:autoduality}
$\rho$ is an isomorphism $\oJ\iso\oPic^0(\oJ)$. Moreover,
$\oPic^0(\oJ)\subset\Pic(\oJ)^=$ is a connected component.
\end{THEOREM}

\begin{remark*} The first statement of the theorem follows immediately from Theorem~\ref{th:oP} using \cite[Theorem~2.6]{compactified}. (Although \cite[Theorem~2.6]{compactified} is formulated for curves with double singularities,
the same argument works in the case of planar singularities if we use \cite[Theorem~C]{Ar}.) The second statement relies on Theorem~\ref{th:FM}. 
\end{remark*}

Finally, we show that $\oP$ also provides a `categorical autoduality' of $\oJ$ in the sense that the corresponding Fourier-Mukai
functor is an equivalence of categories.
\begin{THEOREM}\label{th:FM} Let $D^b(\oJ)$ be the bounded derived category of quasicoherent sheaves on $\oJ$.
The Fourier-Mukai functor
$${\mathfrak F}:D^b(\oJ)\to D^b(\oJ):\cG\mapsto Rp_{1,*}(p_2^*(\cG)\otimes\oP)$$
is an equivalence of categories.  Its quasi-inverse is given by
$$D^b(\oJ)\to D^b(\oJ):\cG\mapsto Rp_{2,*}(p_1^*(\cG)\otimes\oP^\vee)\otimes\det(H^1(C,O_C))^{-1}[g].$$
Here $\oP^\vee=\HOM(\oP,O_{\oJ\times\oJ})$.
\end{THEOREM}

\begin{remarks*}
\begin{myenum}
\item\label{rm:cubic} These results are known in the case of singular plane cubics (nodal or cuspidal). 
The sheaf $\overline P$ is constructed by E.~Esteves and S.~Kleiman in \cite{compactified} for curves $C$ with any
number of nodes and cusps; they also prove that $\overline{P}$ is universal (the first statement of Theorem~\ref
{th:autoduality}). If $C$ is a singular plane cubic, Theorem~\ref{th:FM} is proved by I.~Burban and B.~Kreussler (\cite{BK1}).
J.~Sawon proves it for nodal or cuspidal curves of genus two in \cite{Sawon}.

\item
For simplicity, we consider a single curve $C$ in this section, but all our results hold for
families of curves. Actually, the universal family of curves is used in the proof of Theorem~\ref{th:FM}.
\end{myenum}
\end{remarks*}

\subsection{Organization} The rest of the paper is organized as follows.

Sections~\ref{sc:CM} and \ref{sc:PHilbS} contain preliminary results on Cohen-Macaulay sheaves and punctual
Hilbert schemes of surfaces. These are used in the construction of the Poincar\'e sheaf $\oP$ 
in Section~\ref{sc:Q}. The proof of the key step in the construction is contained in Section~\ref{sc:proof}. 
This completes the proof of Theorem~\ref{th:oP}.

Theorem~\ref{th:oP} easily implies certain simple properties of $\oP$ that are given in Section~\ref{sc:properties}.
In Section~\ref{sc:FM}, we derive Theorem~\ref{th:FM} from these properties, and
show that Theorem~\ref{th:FM} implies Theorem~\ref{th:autoduality}.

\section{Cohen-Macaulay sheaves}\label{sc:CM}
Our argument is based on certain properties of Cohen-Macaulay sheaves. 
Let us summarize these properties.

Let $X$ be a scheme (all schemes are assumed to be of finite type over $\kkk$). 
Denote by $D^b_{coh}(X)\subset D^b(X)$ the coherent derived category of $X$. Suppose that $X$ has pure dimension.
Let us normalize the dualizing complex $\DC_X\in D^b_{coh}(X)$ by the condition that its stalk at generic points of $X$ is 
concentrated in cohomological degree $0$. If $X$ is Gorenstein, $\DC_X$ is an invertible sheaf, which we denote $\omega_X$.

Consider the duality functor
\[\D:D^b_{coh}(X)\to D^b_{coh}(X):\cG\mapsto R\HOM(\cG,\DC_X).\]
Let $M$ be a coherent sheaf on $X$. Set $d:=\codim(\supp(M))$. Then $H^i(\D(M))=0$ for $i<d$. Recall that $M$ is 
Cohen-Macaulay (of codimension $d$) if and only if $H^i(\D(M))=0$ for all $i\ne d$, so that $\D(M)[d]$ is a coherent sheaf.

\subsection{Families of Cohen-Macaulay sheaves}
\begin{lemma} \label{lm:CMFamily}
Let $X$ and $Y$ be schemes of pure dimension, and suppose that $Y$ is Cohen-Macaulay. Suppose that a coherent
sheaf $M$ on $X\times Y$ is flat over $Y$, and that for every point $y\in Y$, the restriction $M|_{X\times\{y\}}$ is
Cohen-Macaulay of some fixed codimension $d$. 
\begin{enumerate}
\item\label{lm:CMFamily1} $M$ is Cohen-Macaulay of codimension $d$.
\item\label{lm:CMFamily2} If in addition $Y$ is Gorenstein, then $\D(M)[d]$ is also flat over $Y$, and
\[\D(M)[d]|_{X\times\{y\}}\simeq \D(M|_{X\times\{y\}})[d]\]
for all points $y\in Y$. 
\end{enumerate}
\end{lemma}
\begin{proof} For \eqref{lm:CMFamily1}, we need to show that $M_z$ is a Cohen-Macaulay $O_z$-module for all $z\in X\times Y$. Set $y=p_2(z)\in Y$. The claim then follows from \cite[Corollary~6.3.3]{EGAIV} applied to morphism $O_y\to O_z$, the $O_y$-module
$O_y$ and the $O_z$-module $M_z$. 

\eqref{lm:CMFamily2} follows from the 
identity
\[L\iota^*(R\HOM(M,\DC_{X\times Y}))=R\HOM(L\iota^*M,L\iota^*\DC_{X\times Y}),\] where 
$\iota:X\times\{y\}\hookrightarrow X\times Y$ is the embedding.
\end{proof}

\begin{remark*} In principle, the lemma can be stated for all (not necessarily closed) points $y\in Y$; then the restriction
$M|_{X\times\{y\}}$ should be understood as an appropriate inverse image. This form of Lemma~\ref{lm:CMFamily} 
is more natural, and it suffices for our purposes. On the other hand, it is easy to see that it suffices to check the Cohen-Macaulay property of $M_z$ for closed points $z\in X\times Y$.
\end{remark*}

\subsection{Extension of Cohen-Macaulay sheaves}
Recall that a Cohen-Macaulay sheaf is \emph{maximal} if it has codimension zero.
Maximal Cohen-Macaulay sheaves are normal in the sense that their sections extend across subsets of codimension two.

\begin{lemma}\label{lm:CMExt} As before, let $X$ be a scheme of pure dimension. Let $M$ be a maximal Cohen-Macaulay sheaf 
on $X$. Then for
any closed subset $Z\subset X$ of codimension at least two, we have $M=j_*(M|_{X-Z})$ for the embedding $j:X-Z\hookrightarrow X$.
\end{lemma}
\begin{proof} This is a special case of \cite[Theorem~5.10.5]{EGAIV}. (Actually, it suffices to require that $M$ has property $(S_2)$.)
\end{proof}

\subsection{Acyclicity}
\begin{lemma} \label{lm:CMAcyclic} Let $f:Y\to X$ be a morphism of schemes. Suppose that $X$ is a Gorenstein scheme of
pure dimension, and that $f$ has finite Tor-dimension. Let $M$ be a maximal Cohen-Macaulay sheaf on $X$.
\begin{enumerate}
\item\label{lm:CMAcyclic1} $Lf^*M=f^*M$.
\item\label{lm:CMAcyclic2} In addition, suppose $Y$ is Cohen-Macaulay. Then $f^*M$ is maximal Cohen-Macaulay.
\end{enumerate}
\end{lemma}
\begin{proof} The statement is local on $X$, so we may assume that it is affine without losing generality. Essentially, the 
statement follows because $M$ is an $\infty$-syzygy sheaf. Indeed, we can include $M$ into a short exact sequence
\[0\to M\to E\to M'\to 0,\]
for some vector bundle $E$ and a maximal Cohen-Macaulay sheaf $M'$. Then 
\[L_if^*M\simeq L_{i+1}f^*M'\quad\text{for all }i>0,\] and \eqref{lm:CMAcyclic1}
follows by induction. 

If $Y$ is Cohen-Macaulay, we can assume that it is of pure dimension (since this is true locally). Then a similar argument 
shows that
\[\EXT^i(f^*M,\DC_Y)\simeq\EXT^{i+1}(f^*M',\DC_Y)\quad\text{for all }i>0,\] which
implies \eqref{lm:CMAcyclic2}. 
\end{proof}

\section{Punctual Hilbert schemes of surfaces}\label{sc:PHilbS}
Let $S$ be a smooth surface. Let us review some properties of the Hilbert scheme of $S$. Fix an integer $n>0$.

\subsection{Hilbert scheme of points}\label{sc:Hilb} 
Let $\Hilb_S=\Hilb_S^n$ be the Hilbert scheme of finite subschemes $D\subset S$ of length $n$. It is well known that 
$\Hilb_S$ is smooth of dimension $2n$, and that it is connected if $S$ is connected. This statement is due to 
J.~Fogarty (\cite{Fog}).

Fix a point $0\in S$, and let $\Hilb_{S,0}\subset\Hilb_S$ be the closed subset of $D\in\Hilb_S$ such that $D$ is 
(set-theoretically) supported at $0$. If we fix local coordinates $(x,y)$ at $0$, we can identify $\Hilb_{S,0}$ with the
scheme of codimension $n$ ideals in $\kkk[[x,y]]$.

\begin{lemma}[J.~Brian{\c{c}}on] \label{lm:HilbS0}
$\dim(\Hilb_{S,0})=n-1$; $\Hilb_{S,0}$ has a unique component of maximal dimension.
\end{lemma}
\begin{proof} See \cite{Brian}, \cite{Iarrobino}, \cite{Nakajima}, or \cite{Baranovsky}.
\end{proof}

\begin{remark*} In fact, $\Hilb_{S,0}$ is irreducible. However, we do not need this claim; without it, Lemma~\ref{lm:HilbS0}
is much easier, see \cite[Theorem 2]{Baranovsky}.
\end{remark*}

Consider the symmetric power $\Sym^n S$. We write its elements as $0$-cycles 
\[\zeta=\sum_{x\in S}\zeta_x\cdot x\qquad (\zeta_x\ge 0,\sum_x\zeta_x=n).\]
Set
\[\supp(\zeta):=\{x\in S\st \zeta_x\ne 0\}\qquad (\zeta=\sum_x\zeta_x\cdot x\in\Sym^n S).\]

\begin{corollary} \label{co:HilbS0}
Consider the \emph{Hilbert-Chow morphism} 
\[\HC:\Hilb_S\to\Sym^n S:D\to\sum_{x\in S}(\length_xD)\cdot x.\]
For any $\zeta\in\Sym^n S$, the preimage $\HC^{-1}(\zeta)$ has a unique component of maximal dimension;
its dimension equals $n-|\supp(\zeta)|$. 
\end{corollary}
\begin{proof} The preimage equals $\prod_x\Hilb_{S,x}^{\zeta_x}$.
\end{proof}

Denote by $\Hilb'_S\subset\Hilb_S$ the open subscheme parametrizing $D\in\Hilb_S$ such that $D$ can be embedded into a smooth
curve (which can be assumed to be $\A1$ without loss of generality). Equivalently, $D\in\Hilb'_S$ if and only if the tangent
space to $D$ at every point is at most one-dimensional. 

\begin{lemma}
$\codim(\Hilb_S-\Hilb'_S)\ge 2$.
\end{lemma}
\begin{proof} If $D\in\Hilb_S-\Hilb'_S$, then $|\supp(\HC(D))|\le n-2$, and the lemma follows from Corollary~\ref{co:HilbS0}.
\end{proof}

\subsection{Flags of finite schemes}\label{sc:Flag}
Let $\Flag'_S$ be the moduli space of flags
\[\emptyset=D_0\subset D_1\subset\dots\subset D_n\subset S,\]
where each $D_i$ is a finite scheme of length $i$ and $D_n\in\Hilb'_S$. It is equipped with maps
\[\psi:\Flag'_S\to\Hilb'_S:(\emptyset=D_0\subset D_1\subset\dots\subset D_n)\mapsto D_n\]
and
\[\sigma:\Flag'_S\to S^n:(\emptyset=D_0\subset D_1\subset\dots\subset D_n)\mapsto(\supp(\ker(O_{D_i}\to O_{D_{i-1}})))_{i=1}^n.\]
Moreover, $\Flag'_S$ carries an action of $S_n$.

\begin{example} \label{ex:T}
For $D\in\Hilb'_S$, choose an embedding $t:D\hookrightarrow\A1$. Then $t(D)=Z(f)$ for a monic degree $n$ polynomial
$f\in\kkk[t]$. The fiber $\psi^{-1}(D)$ is identified with the scheme
\[\Flag_f:=\{(t_1,\dots,t_n)\in\A{n}\st f(t)=(t-t_1)\dots(t-t_n)\};\]
the identification sends $(t_1,\dots,t_n)\in\Flag_f$ to the flag
\[\emptyset\subset Z(t-t_1)\subset Z((t-t_1)(t-t_2))\subset\dots\subset Z(f).\]
The group $S_n$ acts on $\Flag_f$ by permuting $t_i$'s.
\end{example}

The following claim is well known.

\begin{proposition} 
\begin{enumerate}
\item $\psi$ is a degree $n!$ finite flat morphism. 

\item
There exists a unique action of $S_n$ on $\Hilb'_S$ such that $\psi$ and $\sigma$ are equivariant. Here $S_n$
acts on $\Hilb'_S$ trivially and on $S^n$ by permutation of factors. 

\item The fiber of $\psi_*(O_{\Flag'_S})$ over every point of $\Hilb'_S$ is isomorphic to the regular representation of
$S_n$.
\end{enumerate}
\end{proposition}
\begin{proof}
Locally on $\Hilb'_S$, we can embed $D\in\Hilb'_S$ into $\A1$; the claim then follows from direct calculation 
(Example~\ref{ex:T}).
\end{proof}

There is also a more explicit description of $\psi_*(O_{\Flag'_S})$. Let $\cD\subset\Hilb_S\times S$ be the universal family of degree $n$ subschemes of $S$. Let 
$h:\cD\to\Hilb_S$ and $g:\cD\to S$ be the restrictions of projections. Notice that $h$ is a finite flat morphism of degree 
$n$. Set $\cA:=h_*O_\cD$. Then $\cA$ is a coherent sheaf of algebras on $\Hilb_S$ that is locally free of rank $n$. Let 
$\cA^\times\subset\cA$ be the subsheaf of invertible elements. We can view $\cA^\times$ as the sheaf of sections of a 
flat abelian group scheme over $\Hilb_S$: the fiber of this group scheme over $D\in\Hilb_S$ is $\kkk[D]^\times$. 
Clearly, $\cA^\times$ acts on $\cA$, and therefore also on the line bundle $\det(\cA)$. 
The action of $\cA^\times$ on $\det(\cA)$ is given by the norm character $\cN:\cA^\times\to O^\times$.

\begin{lemma}\label{lm:Norm} There is a natural $S_n$-equivariant identification
\[\psi_*(O_{\Flag'_S})=\left((\cA|_{\Hilb'_S})^{\otimes n}\right)_\cN,\]
where the lower index $\cN$ denotes the maximal quotient on which $\cA^\times$ acts via the character $\cN$.
\end{lemma}

\begin{remarks*} 
\begin{myenum}
\item Let us describe the map $(\cA|_{\Hilb'_S})^{\otimes n}\to\psi_*(O_{\Flag'_S})$ of Lemma~\ref{lm:Norm}.
Consider the $n$-fold fiber product
\[\cD_n:=\cD\times_{\Hilb_S}\cD\dots\times_{\Hilb_S}\cD=\{(D,s_1,\dots,s_n)\in\Hilb_S\times S^n\st s_i\in D\text{ for all }i\}.\]
The projection $h_n:\cD_n\to\Hilb_S$ is finite and flat of degree $n^n$ over $\Hilb_S$, and $h_{n,*}(O_{\cD_n})=\cA^{\otimes n}$.
Since the image of the map 
\[(\psi,\sigma):\Flag'_S\to \Hilb_S\times S^n\]
is contained in $\cD_n$, we obtain a morphism of sheaves of algebras 
\[h_{n,*}(O_{\cD_n})|_{\Hilb'_S}\to\psi_*(O_{\Flag'_S}).\]

\item Lemma~\ref{lm:Norm} is similar to the  description of $\psi_*(O_{\Flag'_S})$ given in \cite{Haiman}. Namely, 
$\psi_*(O_{\Flag'_S})$ is the quotient of $(\cA|_{\Hilb'_S})^{\otimes n}$ by the kernel of the symmetric form
\[\cA^{\otimes n}\times\cA^{\otimes n}\to\det\cA:(f_1\otimes\dots\otimes f_n,g_1\otimes\dots\otimes g_n)
\mapsto\bigwedge_{i=1}^n(f_i g_i).\]
This identification extends to $\Hilb_S$ and provides a description of the isospectral Hilbert scheme (see 
Section~\ref{sc:ISHilb}). I do not know whether the identification of Lemma~\ref{lm:Norm} also extends to $\Hilb_S$. Such
an extension would provide another formula for the Poincar\'e sheaf.
\end{myenum}
\end{remarks*}

\begin{proof}[Proof of Lemma~\ref{lm:Norm}] Take $D\in\Hilb'_S$, and choose an embedding $t:D\hookrightarrow\A1$. According
to Example~\ref{ex:T}, we have to identify $\kkk[X_f]$ and $(\kkk[D]^{\otimes n})_\cN$. Explicitly, for 
\[f=t^n-a_1t^{n-1}+\dots+(-1)^na_n,\]
we have
\[\kkk[X_f]=k[t_1,\dots,t_n]/(a_1-(t_1+\dots+t_n),\dots,a_n-(t_1\cdots t_n))\]
and
\[\kkk[D]^{\otimes n}=k[t_1,\dots,t_n]/(f(t_1),\dots,f(t_n)).\]
The multiplicative group
$\kkk[D]^\times$ is generated by linear polynomials. The polynomial $a(b-t)$ acts $\kkk[D]^{\otimes n}$ as 
$a^n(b-t_1)\cdots(b-t_n)$, while
$\cN(a(b-t))=a^nf(b)$. Therefore, $(\kkk[D]^{\otimes n})_\cN$
is the quotient of $\kkk[t_1,\dots,t_n]$ modulo the ideal generated by $f(t_1),\dots, f(t_n)$ and $(b-t_1)\dots(b-t_n)-f(b)$
for all $b\in\kkk$ such that $f(b)\ne 0$. The statement follows.
\end{proof}

\subsection{Isospectral Hilbert scheme}\label{sc:ISHilb}
Let us keep the notation of Section~\ref{sc:Flag}. The following result is a part of M.~Haiman's $n!$ Conjecture.

\begin{proposition}\label{pp:IS}
There exists a scheme $\IS_S$ and a commutative diagram
\[
\xymatrix{S^n\ar@{=}[d]& \Flag'_S\ar[l]_-\sigma\ar[r]^-\psi\ar@{^{(}->}[d]&\Hilb'_S\ar@{^{(}->}[d]\\
S^n&\IS_S\ar[l]\ar[r]&\Hilb_S,}
\]
in which the right square is Cartesian and $\IS_S\to\Hilb_S$ is a degree $n!$ finite flat morphism. The scheme
and the diagram are unique up to a unique isomorphism.
\end{proposition}
\begin{proof} The scheme $\IS_S$ can be constructed by extending $\psi_*O_{\Flag'_S}$ to a rank $n!$ sheaf of algebras
on $\Hilb_S$. Since $\codim(\Hilb_S-\Hilb'_S)\ge 2$, such extension is unique if it exists. This implies the uniqueness claim.

The existence claim is local in the \'etale topology on $\Hilb_S$; it therefore reduces to the case $S=\A2$ proved by 
M.~Haiman in \cite{Haiman}.
\end{proof}

Following \cite{Haiman}, we call $\IS_S$ the \emph{isospectral Hilbert scheme} of $S$. We keep the notation $\psi$ and $\sigma$
for the extended morphisms $\IS_S\to\Hilb_S$ and $\IS_S\to S^n$. Finally, note that the action of $S_n$ on $\Flag'_S$ extends
to its action on $\IS_S$ (because $\IS_S$ is unique), and that $\psi$ and $\sigma$ are equivariant.

\begin{remark*} 
In \cite{Haiman}, it is shown that the map $(\psi,\sigma):\IS_S\to\Hilb_S\times S^n$ is an embedding, so that
$\IS_S$ can be defined as the closure of the image of $\Flag'_S$ in $\Hilb_S\times S^n$. We do not use this property. 
This gives us a choice of two possible references for Proposition~\ref{pp:IS}:  the original argument of \cite{Haiman} 
and V.~Ginzburg's paper \cite{Ginzburg}, which provides a construction of $\IS_S$ based on Hodge $D$-modules.
\end{remark*}

\subsection{Remark: stack of finite schemes}\label{sc:Tsch}
Let us define universal versions of $\Flag'_S$ and $\IS_S$.
 
Let $\FS$ be the algebraic stack parametrizing finite schemes of length $n$. Denote by $\FS_1\subset\FS$ the open substack of
schemes $D\in\FS$ that are isomorphic to a closed subscheme of $\A1$. Denote by $\Flag'_{univ}$ the stack of flags
\[(\emptyset=D_0\subset D_1\subset\dots\subset D_n),\]
where $D_i$ is a finite scheme of length $i$ and $D_n\in\FS_1$. The natural morphism $\Flag'_{univ}\to\FS_1$ has an action
of $S_n$, and the map $\psi$ is obtained from it by base change via $\Hilb'_S\to\FS_1$.

The morphism $\Flag'_S\to\Hilb'_S$ is a cameral cover for the group $\GL(n)$ in the sense
of \cite{DG}. Moreover, $\FS_1$ is identified with the stack of cameral covers for $\GL(n)$, and 
$\Flag'_{univ}\to\FS_1$ is the universal cameral cover. 

Consider $\Hilb_{\A1}$ (the Hilbert scheme of finite subschemes $D\subset\A1$ of length $n$).
The natural map $\Hilb_{\A1}\to\FS_1$ is a presentation, so $\FS_1$ is a quotient
of $\Hilb_{\A1}$ by an action of a groupoid. We can identify $\Hilb_{\A1}$ with the affine space of monic degree $n$ polynomials in $\kkk[t]$ (as in Example~\ref{ex:T}). The elements of the groupoid acting on $\Hilb_{\A1}$ are then interpreted as 
\emph{Tschirnhaus transformations} of polynomials. In this way, the stack $\FS_1$ goes back to 
the seventeenth century \cite{Tsch}.
This relation was pointed out to me by V.~Drinfeld.

Now let $\FS_2\subset\FS$ be the open substack of schemes $D$ that admit an embedding into a smooth surface, which may be
assumed to be $\A2$ without loss of generality. 
The natural morphism $\Hilb_{\A2}\to\FS_2$ is a presentation, and the scheme $\IS_{\A2}$ defined by M.~Haiman descend to
a flat finite stack $\IS_{univ}$ over $\FS_2$. We can view $\IS_{univ}$ 
as the universal isospectral Hilbert scheme: for every smooth
surface $S$, we have 
\[\IS_S=\IS_{univ}\times_{\FS_2}\Hilb_S.\]

\section{Poincar\'e sheaf}\label{sc:Q}
We prove Theorem~\ref{th:oP} by constructing $\oP$. Actually, we construct a sheaf not on $\oJ\times\oJ$
but on its smooth cover, and then show that the sheaf descends to $\oJ\times\oJ$.
This is similar to the construction of automorphic sheaves (\cite{Dr}). 

\subsection{Construction of the Poincar\'e sheaf}
Fix an integer $n>0$, and let $\Hilb_C$ be the Hilbert scheme of finite subschemes $D\subset C$ of degree $n$.
Recall that $\Hilb_C$ is an irreducible locally complete intersection of dimension $n$ (\cite{irreducibility}).
For $n\gg 0$, $\Hilb_C$ is a smooth cover of $\oJ$. More precisely, fix a smooth point $p_0\in C$. It
defines an Abel-Jacobi map
\[\alpha:\Hilb_C\to\oJ:D\mapsto \cI_D^\vee(-np_0)=\HOM(\cI_D,O_C(-np_0)).\]
Here $\cI_D$ is the ideal sheaf of $D\subset C$. For $n\gg0$, the map $\alpha:\Hilb_C\to\oJ$ is smooth and 
surjective. 

Our goal is to construct a sheaf $Q$ on $\Hilb_C\times\oJ$, and then show that it descends to a sheaf $\oP$
on $\oJ\times\oJ$ when $n\gg 0$. The construction of $Q$ makes sense even if $n$ is not assumed to be large.

Let $\cF$ be the universal sheaf on $C\times\oJ$. Thus, for every $F\in\oJ$, the restriction $\cF|_{C\times\{F\}}$
is identified with $F$. We normalize $\cF$ by framing it over $p_0$, so we have
\[\cF|_{\{p_0\}\times\oJ}=O_{\{p_0\}\times\oJ}.\]
Now consider the sheaf 
\[\cF_n:=p_{1,n+1}^*\cF\otimes\dots\otimes p_{n,n+1}^*\cF\]
on $C^n\times\oJ$. In other words, $\cF_n$ is the family of sheaves $F^{\boxtimes n}$ on $C^n$ parametrized by 
$F\in\oJ$. The sheaf $\cF_n$ is $S_n$-equivariant in the obvious way.

Choose a closed embedding $\iota:C\hookrightarrow S$ into a smooth surface $S$, and consider the diagram
\[
\xymatrix{
\Hilb_S\times\oJ\ar[d]^-{p_1}&\IS_S\times\oJ\ar[l]_-{\psi\times\id_\oJ}\ar[r]^-{\sigma\times\id_\oJ}&S^n\times\oJ&C^n\times\oJ
\ar[l]_-{\iota^n\times\id_\oJ}\\
\Hilb_S.}
\]
Set
\begin{equation}\label{eq:Q}
Q:=\left((\psi\times\id_\oJ)_*(\sigma\times\id_\oJ)^*(\iota^n\times \id_\oJ)_*\cF_n\right)^{sign}\otimes p_1^*\det(\cA)^{-1}.
\end{equation}
Here the upper index `sign' stands for the space of anti-invariants with respect to the action of $S_n$. 
Recall that $\det\cA$ is the line bundle on $\Hilb_S$ whose fiber over $D\in\Hilb_S$ is $\det(\kkk[D])$.

Note that \eqref{eq:Q} defines a sheaf on $\Hilb_S\times\oJ$. Let us identify $\Hilb_C$ with a closed subscheme of 
$\Hilb_S$ using $\iota$. The following claim shows that 
\[\supp(Q)\subset\Hilb_C\times\oJ.\] 
This is not immediate, because $\psi(\sigma^{-1}(\iota(C)^n))\not\subset\Hilb_C$.

\begin{proposition}\label{pp:FCM} 
As above, $\iota:C\hookrightarrow S$ is a closed embedding of a reduced curve $C$ into a smooth 
surface $S$ (it is not necessary to assume that $C$ is projective or irreducible). Let $F$ be a torsion-free sheaf of
generic rank one on $C$, and consider the sheaf $(\iota_*F)^{\boxtimes n}$ on $S^n$. Then
\begin{enumerate}
\item\label{pp:FCM1} $L\sigma^* (\iota_*F)^{\boxtimes n}=\sigma^*(\iota_*F)^{\boxtimes n}$;
\item\label{pp:FCM2} $\sigma^*(\iota_*F)^{\boxtimes n}$ is Cohen-Macaulay of codimension $n$;
\item\label{pp:FCM3} $\psi_*(\sigma^*(\iota_*F)^{\boxtimes n})^{sign}$ is supported by the subscheme $\Hilb_C\subset\Hilb_S$. 
\end{enumerate}
\end{proposition} 

It is also not hard to check that $Q$ given by formula \eqref{eq:Q} agrees with $P$ in the following sense.
Let $\Hilb'_C\subset\Hilb_C$ (resp. $\Hilb''_C\subset\Hilb_C$) be the open subscheme parametrizing $D\in\Hilb_C$ such that
$D$ is isomorphic to a subscheme of $\A1$ (resp. $D$ is reduced and contained in the smooth locus $C^{sm}\subset C$).
Note that $\Hilb'_C=\Hilb'_S\cap\Hilb_C$ and $\Hilb''_C\subset\Hilb'_C$. Also, note that $\Hilb''_C$ is dense in $\Hilb_C$
(because $\Hilb_C$ is irreducible) and $\Hilb'_C\subset\Hilb_C$ is a complement of subset of codimension at least two
(this is easy to see from Corollary~\ref{co:HilbS0}).

\begin{lemma}\label{lm:QandP} The restrictions of the sheaves $Q$ and $(\alpha\times\id_\oJ)^*P$ to
$(\Hilb''_C\times\oJ)\cup(\Hilb'_C\times J)$ are naturally isomorphic.
\end{lemma}

We postpone the proof of Proposition~\ref{pp:FCM} until Section~\ref{sc:proof}; Lemma~\ref{lm:QandP} follows from
Proposition~\ref{pp:QandQ'} below. Let us show that Proposition~\ref{pp:FCM} and Lemma~\ref{lm:QandP}
imply Theorem~\ref{th:oP}.

\begin{proof}[Proof of Theorem~\ref{th:oP}]
Proposition~\ref{pp:FCM}\eqref{pp:FCM1} implies that $Q$ is flat over $\oJ$. By 
Proposition~\ref{pp:FCM}\eqref{pp:FCM2},
$Q$ is a flat $\oJ$-family of Cohen-Macaulay sheaves of codimension $n$ on $\Hilb_S$. Therefore, $Q$ is a 
Cohen-Macaulay sheaf of codimension $n$ on $\Hilb_S\times\oJ$ by Lemma~\ref{lm:CMFamily}.

By Proposition~\ref{pp:FCM}\eqref{pp:FCM3}, the restriction $Q|_{\Hilb_S\times\{F\}}$ is supported by
$\Hilb_C\times\{F\}$ for every $F\in\oJ$. Therefore, $Q$ is a maximal Cohen-Macaulay sheaf on $\Hilb_C\times\oJ$.

Finally, consider the Abel-Jacobi map $\alpha:\Hilb_C\to\oJ$ for $n\gg 0$. By Lemma~\ref{lm:QandP}, $Q$ coincides 
with the pullback $(\alpha\times\id_\oJ)^*P$ on the complement of a subset of codimension at least two. It follows
that $Q$ descends to $\oJ\times\oJ$: we can extend the descent data across codimension two using Lemma~\ref{lm:CMExt}.
We thus obtain a sheaf $\oP$ on $\oJ\times\oJ$. It is clear that $\oP$ has the properties required by 
Theorem~\ref{th:oP}.
\end{proof}

\subsection{Restriction to $\Hilb'_S\subset\Hilb_S$}\label{sc:QonHilb'}
The rest of this section contains comments on the formula \eqref{eq:Q}.

Recall that $\Hilb'_S\subset\Hilb_S$ is the open subscheme of $D\in\Hilb_S$ such that $D$ is isomorphic to a
subscheme of $\A1$. Over $\Hilb'_S$, we can identify $\IS_S$ with the space of flags $\Flag'_S$. In this way,
\eqref{eq:Q} is more explicit if we are only interested in the restriction 
$Q|_{\Hilb'_S\times\oJ}$. To make the formula more concrete, let us fix $F\in\oJ$. 

\begin{lemma}\label{lm:QviaFlag'} Consider the diagram
\[
\xymatrix{
\Hilb'_S&\Flag'_S\ar[l]_-\psi\ar[r]^-\sigma&S^n&C^n\ar[l]_-{\iota^n}.}
\]
For every $F\in\oJ$, we have
\[
Q|_{\Hilb'_S\times\{F\}}=\left(\psi_*\sigma^*(\iota^n)_* F^{\boxtimes n}\right)^{sign}\otimes \det(\cA)^{-1}.
\]
\end{lemma}
\begin{proof} Clear.
\end{proof}

We can also rewrite the formula for $Q|_{\Hilb'_S\times\oJ}$ using Lemma~\ref{lm:Norm}. Use the 
diagram
\[
\xymatrix{
\Hilb_S\times\oJ\ar[d]^-{p_1}&\cD\times\oJ\ar[l]_-{h\times\id_\oJ}\ar[r]^-{g\times\id_\oJ}&S\times\oJ&C\times\oJ
\ar[l]_-{\iota\times\id_\oJ}\\
\Hilb_S}
\]
to define the sheaf
\[Q':=\left((\bigwedge\nolimits^n (h\times\id_\oJ)_*(g\times\id_\oJ)^*(\iota\times\id_\oJ)_*\cF))\otimes p_1^*\det(\cA)^{-1}\right)_{p_1^{-1}(\cA^\times)}.\]
Recall that $\cF$ is the universal sheaf on $C\times\oJ$; see Section~\ref{sc:Flag} for definitions of the
remaining objects. Explicitly, the fiber of $Q'$ over $(D,F)\in\Hilb_S\times\oJ$ equals
\[Q'_{(D,F)}=\left((\bigwedge\nolimits^n H^0(D,\iota_*F))\otimes(\det\kkk[D])^{-1}\right)_{\kkk[D]^\times}.\]
Up to the twist by $\det\kkk[D]^{-1}$, the fiber is the largest quotient of $\bigwedge^n H^0(D,\iota_*F)$
on which $\kkk[D]^\times$ acts by the norm character. (Notice the similarity with \cite[Lemma~3]{Dr}.)

\begin{proposition} \label{pp:QandQ'}
The restrictions of $Q$ and $Q'$ to $\Hilb'_S\times\oJ$ are naturally isomorphic.
\end{proposition}
\begin{proof} Follows from Lemma~\ref{lm:Norm}.
\end{proof} 

Lemma~\ref{lm:QandP} follows immediately from Proposition~\ref{pp:QandQ'}.

\subsection{Curves with double singular points} As explained in Section~\ref{sc:QonHilb'}, the restriction
$Q|_{\Hilb'_S\times\oJ}$ is more explicit than $Q$ itself. 

Suppose that all singularities of $C$ are at most double points. (So that at every point $c\in C$, there exists 
$f\in O_c$ such that $\length(O_c/fO_c)=2$.) 

\begin{proposition} \label{pp:alpha'}
For $n\gg 0$, the morphism $\alpha:\Hilb'_C\to\oJ$ is surjective.
\end{proposition}
\begin{proof} 
Let $C^{sing}$ be the singular locus of $C$. For every $p\in C^{sing}$, choose an invertible 
subsheaf $\cI^{(2)}_p\subset O_C$ of degree $-2$ such that $\cI^{(2)}_p\supset\cI_p^2$. Such $\cI^{(2)}_p$ 
exists because $C$ has only double singular points.

Consider $D\in\Hilb_C$. Then $D\in\Hilb'_C$ if and only if $\cI_D\not\subset\cI_p^2$ for every $p\in C^{sing}$. In particular,
$D\in\Hilb'_C$ provided $\cI_D\not\subset\cI^{(2)}_p$ for all $p\in C^{sing}$. 

Now take $F\in\oJ$ and $n\ge 2g+1$. Recall that $p_0\in C$ is the smooth point used to define the Abel-Jacobi map
$\alpha:\Hilb_C\to\oJ$.
Choose a non-zero morphism $\phi:O(-np_0)\to F$. By the Riemann-Roch Theorem, the space of such morphisms $\phi$ has dimension
$n-g+1$. Then $F=\alpha(D)$, where $\cI_D\subset O_C$
is the image of 
\[\phi^\vee:F^\vee(-np_0)\to O_C.\]  
For fixed $p\in C^{sing}$, the space of morphisms $\phi$ such that $\phi(O(-np_0))\subset F\otimes\cI^{(2)}_p$ has dimension
$n-g-1$ by the Riemann-Roch Theorem. Thus, if $\phi$ is generic, we have $\phi(O(-np_0))\not\subset F\otimes\cI^{(2)}_p$ for all
$p\in C^{sing}$, and $D\in\Hilb'_C$.
\end{proof}

By Proposition~\ref{pp:alpha'}, we see that the sheaf $\oP$ on $\oJ\times\oJ$ can be constructed as the descent of
$Q|_{\Hilb'_S\times\oJ}$, assuming $C$ has at most double singular points. Thus, in this case it is possible to describe
$\oP$ without using isospectral Hilbert schemes. 

\begin{remark*} Suppose the singularities of $C$ are arbitrary planar, and denote by $\oJ'\subset\oJ$ the image
$\alpha(\Hilb'_S)$ for $n\gg 0$. It is easy to see that for $n\gg 0$, the image does not depend on $n$ or the choice of $p_0$. The restriction $\oP|_{\oJ'\times\oJ}$ can be constructed without using isospectral Hilbert schemes.
\end{remark*}

\subsection{Independence of embedding} The definition of $Q$ involves the embedding $\iota:C\to S$, and the argument
of Section~\ref{sc:proof} relies on the properties of the Hilbert scheme $\Hilb_S$. On the other hand, it is not hard
to see that the restriction $Q|_{\Hilb_C\times\oJ}$ is independent of this embedding. By Proposition~\ref{pp:FCM}\eqref{pp:FCM3}
this restriction coincides with $Q$.

Let us provide a formula for $Q|_{\Hilb_C\times\oJ}$. Set $\IS_C:=\psi^{-1}\Hilb_C\subset\IS_S$. The morphisms 
$\psi:\IS_C\to\Hilb_C$ and $\sigma:\IS_C\to C^n$ are $S_n$-equivariant for the natural actions, and $\psi$ is a degree $n!$ 
finite flat morphism. It is easy to see that $\IS_C$ does not depend on $\iota:C\to S$. 
Indeed, $\IS_C$ is obtained
from $\IS_{univ}$ by base change $\Hilb_C\to\FS_2$ (see Section~\ref{sc:Tsch}). For another explanation, note that the preimage
$\psi^{-1}(\Hilb'_C)$ is identified with the moduli space $\Flag'_C$ of flags of finite subschemes in $C$, and $\IS_{univ}$
can be viewed as its extension to a finite flat scheme over $\Hilb_C$. Such an extension is unique because 
$\codim(\Hilb_C-\Hilb'_C)\ge 2$ and $\Hilb_C$ is Gorenstein.

Now consider the diagram
\[
\xymatrix{
\Hilb_C\times\oJ\ar[d]^-{p_1}&\IS_C\times\oJ\ar[l]_-{\psi\times\id_\oJ}\ar[r]^-{\sigma\times\id_\oJ}&C^n\times\oJ\\
\Hilb_C.}
\]
Clearly,
\[
Q|_{\Hilb_C\times\oJ}:=\left((\psi\times\id_\oJ)_*(\sigma\times\id_\oJ)^*\cF_n\right)^{sign}\otimes p_1^*\det(\cA)^{-1}.
\]

Similarly, one can describe the restriction $Q|_{\Hilb'_C\times\oJ}$ without choosing $\iota:C\hookrightarrow S$ by rewriting 
the formulas of Section~\ref{sc:QonHilb'}. We leave this description to the reader.

\subsection{Poincar\'e sheaf and automorphic sheaves}\label{sc:langlands}
As we already mentioned, the formula for $\oP$ can be interpreted as a classical limit of V.~Drinfeld's
formula for automorphic sheaves for $\GL(2)$ (\cite{Dr}). Let us sketch the relation. 

Let $X$ be a smooth projective absolutely irreducible curve over a finite field, and let $\cE$ be a geometrically
irreducible $\ell$-adic local system on $X$. In \cite{Dr}, V.~Drinfeld constructs an automorphic 
perverse sheaf $\Aut_\cE$ on the moduli stack $\Bun_2$ of rank two vector bundles on $X$. 

We now apply two `transformations' to this construction. Firstly, let us assume that $X$ is a 
(smooth projective connected) curve over a field $\kkk$ of characteristic zero, which we assume to be algebraically closed
for simplicity. Also, let us replace perverse sheaves with $D$-modules. Now the input of the construction is a rank two
bundle with connection $\cE$ on $X$, and its output is a $D$-module $\Aut_\cE$ on $\Bun_2$.

The second transformation is a `classical limit' (`classical' here refers to the relation between classical 
and quantum mechanics). This involves replacing $D$-modules on a smooth variety (or stack) $Z$ with $O$-modules on the
cotangent bundle $T^*Z$. Now the input $\cE$ of the construction is a rank two Higgs bundle on $X$; equivalently,
$\cE$ is an $O$-module on $T^*X$ whose direct image to $X$ is locally free of rank two. The output is an $O$-module 
$\Aut_\cE$ on $T^*\Bun_2$.

In particular, let $\iota:C\hookrightarrow T^*X$ be an irreducible reduced \emph{spectral curve} for $\GL(2)$: that is, the map 
$C\to X$ is finite of degree two. Note that $C$ has at most double singular points.
We can then take $\cE=\iota_*F$ for a torsion-free sheaf $F$ on $C$ of generic rank one. 
By interpreting Drinfeld's construction in these settings, we obtain a formula for a sheaf $\Aut_\cE$ on $T^*\Bun_2$ 
(more precisely, on the cotangent space to a smooth cover of $\Bun_2$). 
From the point of view of geometric Langlands program, it is natural
to expect that the sheaf is supported on the compactified Jacobian $\oJ$ of $C$, which is embedded in $T^*\Bun_2$
as a fiber of the Hitchin fibration. Thus, given $F\in\oJ$, we have a conjectural construction of a sheaf on $\oJ$. 
This is the construction of $\oP_F$ provided by  Lemma~\ref{lm:QviaFlag'}, with surface $S$ being $T^*X$. 

From this point of view, the general formula \eqref{eq:Q} is obtained by extending Lemma~\ref{lm:QviaFlag'} from
$\Hilb'_S$ to $\Hilb_S$. Presumably, it is the classical limit of the formula for automorphic sheaves for $\GL(n)$
suggested by G.~Laumon (\cite{Lau1,Lau2}) and proved by E.~Frenkel, D.~Gaitsgory, and K.~Vilonen (\cite{FGV,Vanishing}).

Although \eqref{eq:Q} is inspired by results in the area of geometric Langlands conjecture, it is not clear whether the 
proofs from \cite{Dr,FGV,Vanishing} can be adapted to our settings. The argument of
this paper is based on different ideas.

\section{Proof of Theorem~\ref{th:oP}}\label{sc:proof}
It remains to prove Proposition~\ref{pp:FCM}. 
\begin{proof}[Proof of Proposition~\ref{pp:FCM}] 
The morphism $\sigma:\IS_S\to S^n$ has finite Tor-dimension, because $S^n$ is smooth. The subvariety
$C^n\subset S^n$ is locally a complete intersection of codimension $n$. By Corollary~\ref{co:HilbS0},
its preimage $\sigma^{-1}(C^n)\subset\IS_S$ also has codimension $n$. In other words, $C^n\subset S^n$
is locally cut out by a length $n$ regular sequence of functions on $S^n$, and the pull-back of this regular
sequence to $\IS_S$ remains regular. This implies that the restriction $\sigma:\sigma^{-1}(C^n)\to C^n$
has finite Tor-dimension. Now \eqref{pp:FCM1} follows from Lemma~\ref{lm:CMAcyclic}\eqref{lm:CMAcyclic1}.

Recall that $\IS_S$ is Cohen-Macaulay (it is finite and flat over $\Hilb_S$, which is smooth). 
As we saw, $\sigma^{-1}(C^n)\subset\Flag'_S$ is locally a complete intersection. Therefore, $\sigma^{-1}(C^n)$ is also Cohen-Macaulay. Lemma~\ref{lm:CMAcyclic}\eqref{lm:CMAcyclic2} implies \eqref{pp:FCM2}.

Let us prove \eqref{pp:FCM3}. Since $\psi$ is finite and flat, it follows that $\psi_*(\sigma^*(\iota_* F)^{\boxtimes n})$ is a 
Cohen-Macaulay sheaf of codimension $n$. The same is true for its direct summand 
\[M:=\psi_*(\sigma^*(\iota_* F)^{\boxtimes n})^{sign}.\]
Clearly, $\supp(M)\subset \psi(\sigma^{-1}(C^n))$, which is a reducible scheme of dimension $n$. One of its irreducible
components is $\Hilb_C$, and we need to show that $M$ is supported by this component.

Let us first verify this on the level of sets. 
Set \[Z:=\Hilb'_S\cap \psi(\sigma^{-1}((C^{sm})^n)).\] 
Corollary~\ref{co:HilbS0} implies that $\dim(\psi(\sigma^{-1}(C^n))-Z)<n$. Since $M$ is Cohen-Macaulay,
it suffices to check that 
\[\supp(M|_Z)\subset\Hilb_C\cap Z.\] 
But this follows from Proposition~\ref{pp:QandQ'}. Indeed, for $D\in Z$, we have
\[H^0(D,\iota_* F)\simeq H^0(D,\iota_* O_C),\]
because $F$ and $O_C$ are isomorphic in a neighborhood of $D$. Therefore, if $D\not\subset C$, we have
$\dim H^0(D,\iota_* F)<n$, and $M_D=0$ by Proposition~\ref{pp:QandQ'}.

Thus, $M$ is supported by $\Hilb_C$ in the set-theoretic sense. Note that $\psi(\sigma^{-1}(C^n))$ is reduced at the generic point of $\Hilb_C$ (in fact, it contains $\Hilb''_C$
as an open set). Since $M$ is Cohen-Macaulay, we see that its support is reduced, as required.
\end{proof}

\begin{remarks*}
\begin{myenum}
\item Suppose that $S$ admits a symplectic form 
(in fact, Proposition~\ref{pp:FCM} is local on $S$, so we can make this 
assumption without losing generality). Fix a symplectic form on $S$; it is well known that it
induces a symplectic form on $\Hilb_S$. 
One can check that the image $\psi(\sigma^{-1}(C^n))\subset\Hilb_S$ is Lagrangian. (Here one can assume that $C$ is smooth, in which
case the observation is due to I.~Grojnowski \cite[Proposition~3]{Groj}.)
This provides
a conceptual explanation why its dimension equals $n$. 

\item
The irreducible components of $\psi(\sigma^{-1}(C^n))$ have the following description (also contained in \cite{Groj}). 
Consider the 
`diagonal stratification' of $C^n$. The strata are indexed by the set $\Sigma$ of all equivalence relations 
$\sim$ on the set $\{1,\dots,n\}$ (in other words, $\Sigma$ is the set of partitions 
of the set $\{1,\dots,n\}$ into disjoint subsets). Given $\sim\in\Sigma$, 
the corresponding stratum is
\[
C^n_\sim:=\{(x_1,\dots,x_n)\in C^n\st x_i=x_j \text{ if and only if }i\sim j\}.\]
In particular, the open stratum $C^n_=$ corresponds to the usual equality relation $=\in\Sigma$.

For every $\sim\in\Sigma$, the preimage $\sigma^{-1}(C^n_\sim)$ is 
irreducible of dimension $n$ (see Corollary~\ref{co:HilbS0}). 
The irreducible components of $\psi(\sigma^{-1}(C^n))$ are of the form
$\overline{\psi(\sigma^{-1}(C^n_\sim))}$; they are indexed by $\Sigma/S_n$ (which is the set of partitions of $n$). 

\item Let us keep the notation of the previous remark. 
It is not hard to check that at the generic
point of the component corresponding to $\sim\in\Sigma$, the fiber of $\psi_*\sigma^*(\iota_*(F)^{\boxtimes n})$ 
is isomorphic to the space of functions on $S_n/S_\sim$ as a $S_n$-module. Here
\[S_\sim:=\{\tau\in S_n\st \tau(i)\sim i\text{ for all }i\}\subset S_n\]
is the subgroup given by the partition $\sim$. In particular, if
$\sim$ is not discrete, the generic fiber has no anti-invariants
under the action of $S_n$. This provides another explanation of Proposition~\ref{pp:FCM}\eqref{pp:FCM3}. 
\end{myenum}
\end{remarks*}

\section{Properties of the Poincar\'e sheaf}\label{sc:properties}
Consider the Poincar\'e sheaf $\oP$ on $\oJ\times\oJ$ provided by Theorem~\ref{th:oP}. 

\subsection{$\oP$ as an extension}
\begin{lemma}\label{lm:oPCM} Let  $j:J\times\oJ\cup\oJ\times J\hookrightarrow\oJ\times\oJ$ be the open embedding.
\begin{enumerate} 
\item\label{lm:oPCM1} $\oP$ is a maximal Cohen-Macaulay sheaf on $\oJ\times\oJ$;
\item\label{lm:oPCM2} $\oP=j_*P$;
\item\label{lm:oPCM3} 
$\oP$ is equivariant with respect to the permutation of the factors $p_{21}:\oJ\times\oJ\to\oJ\times\oJ$.
\end{enumerate}
\end{lemma}
\begin{proof} \eqref{lm:oPCM1} follows from Lemma~\ref{lm:CMFamily}. Now Lemma~\ref{lm:CMExt} implies
\eqref{lm:oPCM2}. Finally, \eqref{lm:oPCM3} is also clear, because $P$ is equivariant under $p_{21}$.
\end{proof}
In particular, Theorem~\ref{th:oP}\eqref{th:oP3} and Lemma~\ref{lm:oPCM}\eqref{lm:oPCM3} imply that $\oP$ is flat with 
respect to both projections $\oJ\times\oJ\to\oJ$, and that for every $F\in\oJ$, the restrictions $\oP|_{\{F\}\times\oJ}$
and $\oP|_{\oJ\times\{F\}}$ give the same sheaf on $\oJ$, which we denoted by $\oP_F$.

\subsection{$\oP$ and duality}
Consider now the duality involution
\[\nu:\oJ\to\oJ:F\to F^\vee:=\HOM(F,O_C).\]
Note that $\nu$ is an algebraic map by Lemma~\ref{lm:CMFamily}.
\begin{lemma}\label{lm:oPvee} 
\begin{enumerate}
\item $(\nu\times\id_\oJ)^*\oP=(\id_\oJ\times\nu)^*\oP=\oP^\vee$;
\item $(\nu\times\nu)^*\oP=\oP$.
\end{enumerate}
\end{lemma}
\begin{proof} By Lemma~\ref{lm:oPCM}, all of the sheaves in the statement are maximal Cohen-Macaulay. It remains to notice
that over $(J\times\oJ\cup\oJ\times J)\subset\oJ\times\oJ$ the required identifications are clear from the definition of $\oP$.
\end{proof}

\begin{corollary} Let ${\mathfrak F}:D^b(\oJ)\to D^b(\oJ)$ be the Fourier-Mukai functor of Theorem~\ref{th:FM}. Its restriction
to $D^b_{coh}(\oJ)$ satisfies
\[{\mathfrak F}\circ\D\simeq(\nu^*\circ\D\circ{\mathfrak F})[-g].\]
\end{corollary}
\begin{proof}
By Serre's duality, the functor $\D\circ{\mathfrak F}\circ\D$ is given by
\[D^b_{coh}(\oJ)\to D^b_{coh}(\oJ):\cG\mapsto Rp_{1,*}(p_2^*(\cG)\otimes\oP^\vee)\otimes\omega_\oJ[-g].\]
Now Lemma~\ref{lm:oPvee} implies the statement.
\end{proof}

\subsection{Theorem of the Square}
Consider the universal Abel-Jacobi map 
\[A:J\times C\to\oJ:(L,c)\mapsto L(c-p_0):=\HOM(\cI_{c},L(-p_0)).\]
(Recall that $p_0\in C$ is a fixed smooth point.) 
It is easy to see that $\oP$ agrees with it in the following sense:
\begin{lemma} \label{lm:AJ}
Consider the diagram
\[
\xymatrix{J\times\oJ& J\times C\times\oJ\ar[r]^-{p_{23}}\ar[l]_-{p_{13}}\ar[d]^-{A\times\id_\oJ}& C\times\oJ\\
&\oJ\times\oJ.}
\]
We have $(A\times\id_\oJ)^*\oP=p_{23}^*(\cF)\otimes p_{13}^*P$. Recall that $\cF$ is the universal sheaf on $C\times\oJ$.
\end{lemma}
\begin{proof}
Both sides are maximal Cohen-Macaulay, and their restrictions to $J\times(C^{sm}\times\oJ\cup C\times J)$ are identified. Here $C^{sm}\subset C$ is the smooth locus of $C$.
\end{proof}

\begin{remark*}
Set $Y:=(\oJ\times C^{sm}\cap J\times C)\subset\oJ\times C$. The map $A:J\times C\to\oJ$ 
extends to a regular map $Y\to\oJ$. Lemma~\ref{lm:AJ} remains true for this extension: it provides an isomorphism
of sheaves on $Y\times\oJ$. 
\end{remark*}

Similarly, we check that $\oP$ satisfies the 
Theorem of the Square. 
Namely, consider the action 
\[\mu:J\times\oJ\to\oJ:(L,F)\to L\otimes F.\]
Clearly, $\mu$ is a smooth algebraic morphism. 

\begin{lemma} \label{lm:Square}
Consider the diagram
\[
\xymatrix{J\times\oJ& J\times\oJ\times\oJ\ar[r]^-{p_{23}}\ar[l]_-{p_{13}}\ar[d]^-{\mu\times\id_\oJ}&\oJ\times\oJ\\
&\oJ\times\oJ.}
\]
We have $(\mu\times\id_\oJ)^*\oP=p_{13}^*(P)\otimes p_{23}^*\oP$.
\end{lemma}
\begin{proof}
Both sides are maximal Cohen-Macaulay, and their restrictions to $J\times(J\times\oJ\cup\oJ\times J)$ are identified.
\end{proof}

\begin{remark*} Lemmas~\ref{lm:AJ} and \ref{lm:Square} are contained, in a form, in \cite{compactified}. Namely, Lemma~\ref{lm:AJ} is contained in the proof of \cite[Theorem~2.6]{compactified}, while Lemma~\ref{lm:Square} is equivalent to \cite[Proposition~2.5]{compactified}. Both statements are formulated for curves with double singularities, but this assumption can be removed using \cite[Theorem~C]{Ar}.
\end{remark*}

\subsection{Hecke eigenproperty} Let us also state a less obvious property of $\oP$, which is motivated by the Langlands program.
Essentially, we claim that $\oP$ is an `eigenobject' with respect to natural `Hecke endofunctors'. The proof of the property will
be given elsewhere; it is not used in this paper.

Let $\Heck$ be the moduli space of collections $(F_1,F_2,f)$,
where $F_1,F_2\in\oJ$ and $f$ is a non-zero map $F_1\hookrightarrow F_2(p_0)$, defined up to scaling. Informally,
$F_2(p_0)$ is an \emph{elementary upper modification} of $F_1$ at the point $\supp(\coker(f))\in C$. The space $\Heck$ 
is equipped with maps $\phi_1,\phi_2:\Heck\to\oJ$ and $\gamma:\Heck\to C$ that send $(F_1,F_2,f)$ to $F_1\in\oJ$, $F_2\in\oJ$,
and $\supp(\coker(f))$ respectively.

The Hecke eigenproperty claims that for every $F\in\oJ$, we have \[R(\phi_1,\gamma)_*\phi_2^*(\oP_F)\simeq\oP_F\boxtimes F.\]
We also have a `unversal' Hecke property for the sheaf $\oP$; its precise statement is left to the reader. The
Hecke eigenproperty generalizes Lemma~\ref{lm:AJ}.

\section{Fourier-Mukai transform}\label{sc:FM}
Set \[\Psi:=Rp_{13*}(p_{12}^*\oP^\vee\otimes p_{23}^*\oP)\in D^b(\oJ\times\oJ).\] 
Our goal is to prove
\begin{proposition} \label{pp:FM}
$\Psi\simeq O_\Delta[-g]\otimes_\kkk\det H^1(C,O_C)$.
\end{proposition}
Proposition~\ref{pp:FM} implies Theorem~\ref{th:FM} by the argument completely analogous to the proof of 
\cite[Theorem~2.2]{Mukai}. The proof of Proposition~\ref{pp:FM} follows the same pattern as the proof of \cite[Theorem~A]{Ar}.

\subsection{Upper bound on the support of $\Psi$}
Denote by $\tilde g$ the geometric genus of $C$.

For any $F\in\oJ$, consider the restriction $P_F=(\oP_F)|_J$. It is a line bundle on $J$. 

\begin{proposition}[{cf. \cite[Proposition~1]{Ar}}] \label{pp:1}
If $(F_1,F_2)\in\supp(\Psi)$, then $P_{F_1}\simeq P_{F_2}$.
\end{proposition}
\begin{proof} By base change, $(F_1,F_2)\in\supp(\Psi)$ if and only if $\HH^i(\oJ,\oP_{F_1}\otimes^L\oP_{F_2}^\vee)\ne 0$
for some $i$. (Note that the hypercohomology may be non-zero even if $i$ is negative.) Let $T_i\to J$ be the $\gm$-torsor
corresponding to $P_{F_i}$ for $i=1,2$. From definition \eqref{eq:Poincare}, we see that $T_i$ is naturally an abelian group which
is an extension of $J$ by $\gm$. Lemma~\ref{lm:Square} implies that the action of $J$ on $\oJ$ lifts to an action of $T_i$
on $P_{F_i}$. 

Now let $T$ be the $\gm$-torsor corresponding to $P_{F_1}\otimes P_{F_2}^\vee$. It is also an extension of $J$ by $\gm$ (the
difference of extensions $T_1$ and $T_2$). It follows that $T$ acts on $\HH^i(\oJ,\oP_{F_1}\otimes^L\oP_{F_2}^\vee)$ for all $i$,
with $\gm\subset T$ acting via the tautological character.

Finally, if $\HH^i(\oJ,\oP_{F_1}\otimes^L\oP_{F_2}^\vee)\ne 0$, it contains an irreducible $T$ submodule $V$. Since $T$ is
abelian, $\dim(V)=1$, and $T$ acts by a character $\chi:T\to\gm$. Then $\chi$ provides a splitting $T\simeq J\times\gm$.
This implies the statement.
\end{proof}

\begin{corollary}\label{co:2}
 If $(F_1,F_2)\in\supp(\Psi)$, then $F_1|_{C^{sm}}\simeq F_2|_{C^{sm}}$.
\end{corollary}
\begin{proof} Same proof as \cite[Corollary~2]{Ar}: pull back the isomorphism of Proposition~\ref{pp:1} by the Abel-Jacobi map
$C\to\oJ$.
\end{proof}

\begin{proposition}[{cf. \cite[Corollary~3]{Ar}}]\label{pp:3} 
Consider the map $\mu\times\id_\oJ: J\times\oJ\times\oJ\to\oJ$. For any $F_1,F_2\in\oJ$,
the intersection
\[Z:=(\mu\times\id_\oJ)^{-1}(\supp(\Psi))\cap J\times\{F_1\}\times\{F_2\}\] 
is of dimension at most $g-\tilde g$.
\end{proposition}
\begin{proof}
By Corollary~\ref{co:2}, if $(L,F_1,F_2)\in Z$, then 
\begin{equation}\label{eq:Csm}
(L\otimes F_1)|_{C^{sm}}\simeq F_2^\vee|_{C^{sm}}.
\end{equation}
Looking at normalization of $C$, it is easy to see that the set of $L$ satisfying \eqref{eq:Csm} is a countable union of
subvarieties of dimension $g-\tilde g$. (In particular, it does not contain generic points of subschemes of higher dimension.)
\end{proof}

\begin{corollary}\label{co:bound} Suppose $C$ is singular, so $\tilde g<g$. Then $\dim(\supp(\Psi))<2g-\tilde g$.
\end{corollary}
\begin{proof} Since $\mu:J\to\oJ\to\oJ$ is smooth of relative dimension $g$, it suffices to show that 
\[\dim(\mu\times\id_\oJ)^{-1}(\supp(\Psi))<3g-\tilde g.\]
Proposition~\ref{pp:3} implies that the fibers of the projection $(\mu\times\id_\oJ)^{-1}(\supp(\Psi))\to\oJ\times\oJ$
have dimension at most $g-\tilde g$, while \cite[Theorem~A]{Ar} shows that over $J\times J$, the fibers are zero-dimensional.
\end{proof}

\subsection{Proof of Theorem~\ref{th:FM}}

\begin{lemma}\label{lm:Serre}
Let $X$ be a scheme of pure dimension. Suppose $\cG\in D^b_{coh}(X)$ satisfies the following conditions:
\begin{enumerate}
\item\label{lm:Serre1} $\codim(\supp(\cG))\ge d$;
\item\label{lm:Serre2} $H^i(\cG)=0$ for $i>0$;
\item\label{lm:Serre3} $H^i(\D\cG)=0$ for $i>d$.
\end{enumerate}
Then $\cG$ is a Cohen-Macaulay sheaf of codimension $d$.
\end{lemma}
\begin{proof} The proof is naturally given in the language of perverse coherent sheaves (\cite{AB}). Indeed, \eqref{lm:Serre3}
claims that $\D\cG\in D^{p,\le0}(X)$ for perversity $p:|X|\to\Z$ given by $p(x)=d$. Here $|X|$ is the set of 
(not necessarily closed) points of $X$.
Therefore, $\cG\in D^{\overline{p},\ge0}(X)$ for the dual perversity 
\[\overline{p}:|X|\to\Z:p(x)=\codim(\overline{\{x\}})-d.\]
Now \eqref{lm:Serre1} implies that $\cG\in D^{\ge 0}(X)$, and \eqref{lm:Serre2} implies that $\cG$ is a sheaf. Since
$\D\cG[d]$ also satisfies conditions \eqref{lm:Serre1}--\eqref{lm:Serre3}, it is also a sheaf, and therefore $\cG$ is 
Cohen-Macaulay of codimension $d$.
\end{proof} 

Note that the statement of Lemma~\ref{lm:Serre} is local in smooth topology on $X$; therefore, the lemma still holds
if $X$ is an algebraic stacks (locally of finite type over $\kkk$).

\begin{proof}[Proof of Proposition~\ref{pp:FM}]
Both $\oP$ and $\Psi$ are defined for families of curves with plane singularities. Let us consider the universal family.
Let $\cM$ be the moduli stack of (reduced irreducible projective) curves of fixed arithmetic genus $g$ with 
plane singularities. Let $\cC\to\cM$ be the universal curve, and let $\overline\cJ$ 
(resp. $\cJ\subset\overline\cJ$) be the relative compactified Jacobian (resp. the relative Jacobian)
of $\cC$ over $\cM$. The properties of these objects are summarized in
\cite{Ar}. As $C\in\cM$ varies, the family of Poincar\'e sheaves gives a Cohen-Macaulay sheaf $\oP_{univ}$ on $\overline\cJ\times_\cM\overline\cJ$; similarly, $\Psi$ is naturally defined as an object of the derived category
\[\Psi_{univ}\in D^b(\overline\cJ\times_\cM\overline\cJ).\]
The restriction of $\Psi_{univ}$ to the fiber over a particular curve $C\in\cM$ is $\Psi$.

Denote by $\frj$ the rank $g$ vector bundle on $\cM$ whose fiber over $C\in\cM$ is $H^1(C,O_C)$. Alternatively, $\frj$
can be viewed as the bundle of (commutative) Lie algebras corresponding to the group scheme $\cJ\to\cM$. 
Denote the projection $\overline\cJ\times_\cM\overline\cJ\to\cM$ by $\pi$ and diagonal in 
$\overline\cJ\times_\cM\overline\cJ$ by $\Delta$. 
Our goal is to prove that
\begin{equation}\label{eq:Psi}
\Psi_{univ}[g]\simeq O_\Delta\otimes\pi^*\det(\frj).
\end{equation}

Consider on $\cM$ the stratification $\cM^{(\tilde g)}$, where $\cM^{(\tilde g)}$ parametrizes curves of geometric
genus $\tilde g$. By \cite[Proposition~6]{Ar}, $\codim(\cM^{(\tilde g)})\ge g-\tilde g$. Now Corollary~\ref{co:bound}
implies that $\dim(\supp(\Psi))=\dim(\cM)+g$, and moreover, every maximal-dimensional component of $\supp(\Psi)$
meets $\pi^{-1}(\cM^{(g)})$. Since $\cM^{(g)}$ parametrizes smooth curves, 
\[\supp(\Psi)\cap\pi^{-1}(\cM^{(g)})=\Delta\cap\pi^{-1}(\cM^{(g)}).\]
This is a reformulation of Mumford's result \cite[Section II.8.(vii)]{Mumford}.

Finally, $H^i(\Psi_{univ})=0$ for $i>g$ and Serre's duality implies that $H^i(\D\Psi_{univ})=0$ for $i>0$. 
(For instance,  we have
\[\D\Psi\simeq p_{21}^*\Psi[g]\]
over a fixed curve $C\in\cM$.)
Thus, Lemma~\ref{lm:Serre} shows that $\Psi_{univ}[g]$ is a Cohen-Macaulay sheaf of codimension $g$. Outside of a set of 
codimension $g+1$, we see that $\supp(\Psi_{univ})$ coincides with $\Delta$, therefore, 
\[\supp(\Psi_{univ})\subset\Delta.\]
Also, \cite[Theorem~10]{Ar} provides the required isomorphism \eqref{eq:Psi}
over $\cJ\times_\cM\cJ$. Thus both sides of \eqref{eq:Psi} are Cohen-Macaulay sheaves on $\Delta$ that are isomorphic outside
of a subset of codimension two. By Lemma~\ref{lm:CMExt}, they are isomorphic.
\end{proof}

As we have seen, Proposition~\ref{pp:FM} implies Theorem~\ref{th:FM}.

\subsection{Autoduality of the compactified Jacobian}
It remains to prove Theorem~\ref{th:autoduality}. As we already mentioned, the first statement follows from Theorem~\ref{th:oP}
and the results of \cite{compactified}. On the other hand, both statements easily follow from Theorem~\ref{th:FM}.

Note that $\Pic(\oJ)^=$ is not claimed to be a fine moduli space; that is, there may be no universal family of torsion-free 
sheaves on $\oJ$ of generic rank one parametrized by $\Pic(\oJ)^=$. However, locally in the \'etale topology of $\Pic(\oJ)^=$, 
such a family exists and is unique (\cite[Theorem~3.1]{CP2}). In particular, points of $\Pic(\oJ)^=$ are in bijection with
isomorphism classes of torsion-free sheaves on $\oJ$ of generic rank one. We will make no distinction between these two objects,
so that $M\in\Pic(\oJ)^=$ means ``$M$ is a torsion-free sheaf on $\oJ$ of generic rank one, defined up to 
non-canonical isomorphism''.

\begin{proof}[Proof of Theorem~\ref{th:autoduality}]
Given $M\in \Pic(\oJ)^=$, consider its Fourier-Mukai transform $\FF(M)$. Its cohomology sheaves
are concentrated in degrees between $0$ and $g$. Fix an ample line bundle on $\oJ$, and denote by $h^i(\FF(M))\in\Q[t]$ the
Hilbert polynomial of the cohomology sheaf $H^i(\FF(M))$ for $0\le i\le g$.

Proposition~\ref{pp:FM} implies that if $M=\rho(F)=\oP_F$ for $F\in\oJ$, then $\FF(M)\simeq O_{F^\vee}[-g]$. Here $O_{F^\vee}$ 
is the structure sheaf of the point $F^\vee\in\oJ$. In particular,
\begin{equation}\label{eq:FMvanishing}
h^i(\FF(M))=\begin{cases}1,& i=g\\0,& i\ne g\end{cases}.
\end{equation}
On the other hand, consider $h^i(\FF(M))$ as functions of $M\in\Pic(\oJ)^=$. 
They are semicontinuous with respect to the order on $\Q[t]$
given by
\[f>g\quad\text{if}\quad f(t)>g(t)\quad\text{for $t\gg0$}\quad(f,g\in\Q[t]).\]
Therefore, \eqref{eq:FMvanishing} also holds for $M$ in a neighborhood of $\rho(\oJ)\subset\Pic(\oJ)^=$.

However, if $M$ satisfies \eqref{eq:FMvanishing}, then $\FF(M)\simeq O_F[-g]$ for some point $F\in\oJ$. Therefore,
$\FF(M)\simeq\FF(\oP_{F^\vee})$, and Theorem~\ref{th:FM} implies $M\simeq\oP_{F^\vee}$. Hence $\rho(\oJ)$ is a connected component
of $\Pic(\oJ)^=$.

We also see that the inverse of the map $\rho:\oJ\to\rho(\oJ)$ is given by 
\[M\mapsto\nu(\supp(\FF(M))):\rho(\oJ)\to\oJ.\]
Clearly, this is an algebraic map. This completes the proof.
\end{proof} 

\bibliographystyle{abbrv}
\bibliography{cjacobians}
\end{document}